\documentclass[letterpaper, 10 pt, conference,onecolumn]{IEEEconf}  

\IEEEoverridecommandlockouts                              
           
\usepackage{amssymb}
\usepackage{eqnarray} 

\usepackage{amsmath}
\usepackage{mathalfa} 
\usepackage{pdfpages}		
\usepackage{dsfont}   		
\usepackage{algorithm}
\usepackage{algpseudocode}

\usepackage{bbm}
                                               
\usepackage{graphicx}
\usepackage[small]{caption}
\usepackage[utf8]{inputenc}
\usepackage{psfrag}
\usepackage{amsmath}
\usepackage{amssymb}
\usepackage{mathtools}

\usepackage{psfrag}
\usepackage{url}

\usepackage{cite}

\newcommand{\sD}{\mathcal{D}}
\newcommand{\sS}{\mathcal{S}}

\newcommand{\sL}{\mathcal{L}}

\newcommand{\sT}{\mathcal{T}}

\newcommand{\R}{\mathbb{R}}
\newcommand{\E}{\mathbb{E}}
\newcommand{\Z}{\mathbb{Z}}

\newcommand{\bv}[1]{\mathbf{#1}}
\newcommand{\policy}[2]{\pi_{#1}^{#2}}

\newcommand{\abs}[1]{\left\vert #1 \right\vert}

\newcommand{\KL}{\text{KL}}

\newcommand{\DKL}{\mathcal{D}_{\KL}}

\newtheorem{Remark}{Remark}
\newtheorem{Theorem}{Theorem}
\newtheorem{Definition}{Definition}
\newtheorem{Lemma}{Lemma}

\newtheorem{Assumption}{Assumption}
\newtheorem{Proposition}{Proposition}
\newtheorem{Problem}{Problem}

\title{\LARGE \bf
On the design of autonomous agents from multiple data sources}

\author{\'Emiland Garrab\'e and Giovanni Russo$^\ast$
\thanks{$^\ast$ E. Garrab\'e (email: {\tt\small egarrabe@unisa.it}) and G. Russo (email: {\tt\small giovarusso@unisa.it}) are with the Department of Information and Electrical Engineering and Applied Mathematics (DIEM) at the University of Salerno, 84084, Salerno, Italy.}%
}

\begin{document}

\maketitle
\thispagestyle{empty}
\pagestyle{empty}

\begin{abstract}
This paper is concerned with the problem of designing agents able to dynamically select information from multiple data sources in order to tackle  tasks that involve tracking a target behavior while optimizing a reward. We formulate this problem as a data-driven optimal control problem with integer decision variables and give an explicit expression for its  solution. The solution determines how (and when) the data from the sources should be used by the agent. We also formalize a notion of agent's regret and, by relaxing the problem, give a regret upper bound. Simulations complement the results.
\end{abstract}

\section{INTRODUCTION}
The problem of designing  agents able to tackle tasks by using data coming from multiple sources is attracting  much research attention. Learning algorithms with multiple simulators in the  loop \cite{7106543}, robots navigating via maps from other robots \cite{5876227} and the sharing economy  \cite{Sharing_book} are a few examples that involve  making decisions based on the crowdsourcing of information \cite{9244209}. For these {crowdsourcing} agents, which  craft  their  actions  from mutually excluding options, a number of key questions arise: {\em is it possible to optimally orchestrate the use of the sources? And can performance be rigorously quantified?} Positively answering these questions would have intriguing implications, highlighting the possibility of designing agents that can optimally perform tasks by re-using {\em different} datasets. In this context, this paper presents an optimal decision making strategy that allows an agent to switch between the sources in order to tackle certain tasks. Moreover, we define and quantify the regret of the agent that uses our strategy.

\subsubsection*{Related Work} in \cite{9244209}, we  formalized the so-called {crowdsourcing control problem} (CCP) as a finite-horizon data-driven optimal control problem. Results on data-driven control include \cite{8795639,8933093,8960476}, which take a behavioral systems perspective, \cite{8703172}, that presents an approach to compute minimum-energy controls for linear systems, \cite{8039204,DBLP:conf/l4dc/WabersichZ20,9109670} which introduce algorithms inspired from MPC, \cite{TANASKOVIC20171}, where a data-driven design approach is devised for stabilizable Lipschitz systems. We also recall here \cite{Gagliardi_D_et_Russo_G_IFAC2020_extended_Arxiv} which, by leveraging an approach that can be traced back to \cite{Karny_M_Automatica_1996_Towards_Fully_Prob}, seeks to synthesize control policies from demonstration datasets for systems with actuation constraints and the recent \cite{9296954}, which discusses some of the key challenges of extracting control relevant information from large amounts of data. Recently, it has been shown that data-driven control can benefit from the use of properly {\em patched} data from multiple sources \cite{9062331}. In the context of learning systems, a Multi-Fidelity Reinforcement Learning (RL) algorithm, enabling an agent to build a policy by sampling from different datasets, has been introduced/analyzed in \cite{7106543} and, in  e.g.  \cite{pmlr-v119-agarwal20c,OFFLINE_RL}, it has been also shown that state-of-the-art RL algorithms can improve their performance if these are trained on multiple, heterogeneous, offline datasets.

\subsubsection*{Contributions}  we consider the problem of designing agents able to dynamically choose between multiple sources in order to fulfill tasks that involve tracking a target behavior while optimizing a reward function. We formulate this problem as a data-driven, integer, optimal control problem. For this problem: (i) we formalize the problem as a data-driven  control problem. This leads to an optimal control formulation with integer decision variables; (ii) by  relaxing the problem, we find an explicit expression for its optimal solution. The solution determines, at each time step, which source should be picked by the agent; (iii) finally, we formalize a notion of regret w.r.t. an oracle and give a regret upper bound. The oracle is obtained by  relaxing the original control problem. This  leads to an infinite-dimensional convex problem that we also solve explicitly. Our results are  complemented via simulations.

While the results of this paper are inspired by the ones presented in \cite{9244209}, our paper extends this in several ways. First, differently from \cite{9244209}, here the formalization of the control problem leads to study an optimization problem with integer decision variables. Second, we find an optimal solution for the problem considered here and actually prove that this solution coincides with a non-optimal solution from \cite{9244209}. Moreover,  we introduce a notion of regret and give a regret upper bound. This leads to an infinite-dimensional convex problem that was not considered in \cite{9244209} and, as a result, our regret analysis cannot obtained with the results from \cite{9244209}.

\subsubsection*{Mathematical Preliminaries}

sets are in {\em calligraphic} and vectors  in {\bf bold}. Consider the measurable space $(\mathcal{X},\mathcal{F}_x)$, where $\mathcal{X}\subseteq\R^{d}$ ($\mathcal{X}\subseteq\Z^{d}$) and $\mathcal{F}_x$ is a $\sigma$-algebra on $\mathcal{X}$. 
We denote a random variable on $(\mathcal{X},\mathcal{F}_x)$  by $\mathbf{X}$ and its realization  by $\mathbf{x}$. The \textit{probability density ({mass})  function} or \textit{pdf} (\textit{pmf}) of a continuous (discrete) $\mathbf{X}$ is denoted by $p(\mathbf{x})$ and $\sD$ is the convex subset of pdfs/pmfs. When we take the integrals (sums) involving pdfs (pmfs) we always assume that these exist. The  expectation of a function $\mathbf{h}(\cdot)$ of the continuous  variable $\mathbf{X}$ is  $\E_{{p}}[\mathbf{h}(\mathbf{X})]:=\int\mathbf{h}(\mathbf{x})p(\mathbf{x})d\mathbf{x}$, where the integral is over the support of $p(\mathbf{x})$ - for discrete variables the integral is replaced with the sum. The \textit{joint} pdf (pmf) of two random variables, $\mathbf{X}_1$ and $\mathbf{X}_2$, is  $p(\mathbf{x}_1,\mathbf{x}_2)$ and the \textit{conditional} pdf (pmf) of $\mathbf{X}_1$ w.r.t. $\mathbf{X}_2$ is $p\left( \mathbf{x}_1| \mathbf{x}_2 \right)$. Countable sets are denoted by $\lbrace w_k \rbrace_{k_1:k_n}$, where $w_k$ is a set element and $k_1:k_n$ is the closed set of consecutive integers between index $k_1$ and $k_n$. We also let $p_{0:N}:= p(\bv{x}_0,\ldots,\bv{x}_N)$, $p_{k:k} := p_k(\bv{x}_k)$ and $p_{k|k-1}:=p_k(\bv{x}_k|\bf{x}_{k-1})$.  \textcolor{black}{Also: (i) functionals are denoted by capital calligraphic characters with arguments in curly brackets; (ii) the internal product between tensors is $\langle \cdot, \cdot \rangle$; (iii) likelihoods have the arguments in square brackets.} The Kullback-Leibler (KL) divergence of $p:=p(\mathbf{x})$ w.r.t. $q:=q(\mathbf{x})$, with $p$ absolutely continuous with respect to $q$, is  $\mathcal{D}_{\KL}\left(p || q \right):= \int p \; \ln\left( {p}/{q}\right)\,d\mathbf{x}$ (for discrete variables the integral is replaced by the sum). For pdfs/pmfs, $\mathcal{D}_{\KL}\left(p^{(1)} || p^{(2)} \right)$ measures the proximity of the pair $p^{(1)}$, $p^{(2)}$. 

\section{The  set-up}\label{sec:problem_formulation}
The agent seeks to craft its behavior by gathering data from a number of sources in order to fulfill a task that involves tracking a target/desired behavior while, at the same time, maximizing a reward function. In what follows, $\bv{x}_k\in\mathcal{X}$ is the state at time step $k$, $\bv{d}_{0:N}:=\{\bv{x}_0,\ldots,\bv{x}_N\}$ is the dataset of the target state trajectory for the agent over $\sT:={0:N}$ and $p(\bv{d}_{0:N}) = p(\bv{x}_0,\ldots,\bv{x}_N)$. Now, by making the standard assumption that the Markov property holds, we have:
\begin{equation}\label{eqn:target}
\begin{split}
p_{0:N} &  = p_{0:0}\prod_{k=1}^N p_{k|k-1}=: p_{0:0}p_{1:N|0},
\end{split}
\end{equation}
which is termed as {\em target behavior}.  As noted in \cite{9244209}, the agent behavior can be designed by designing the joint pdf/pmf $\pi(\bv{x}_0,\ldots,\bv{x}_N)$. By letting $\pi_{k|k-1}:= \pi_k(\bv{x}_k|\bv{x}_{k-1})$ we have:
\begin{equation}\label{eqn:agent}
\begin{split}
\pi_{0:N}  
 &  = 
 \pi_{0:0}\prod_{k=1}^N \pi_{k|k-1} =: \pi_{0:0}\pi_{1:N|0}.
\end{split}
\end{equation}
That is, the {\em agent behavior} $\pi_{0:N} $ can be designed by shaping the $\pi_{k|k-1}$'s. We let $r_k:\mathcal{X}\rightarrow\R$ be the reward obtained by the agent for being in state $\bv{x}_k$ at time $k$. Hence, the expected reward for the agent that follows the behavior in (\ref{eqn:agent}) is $\E_{\pi_{k-1:k-1}}\left[\tilde{r}_k(\bv{X}_{k-1})\right] := \E_{\pi_{k-1:k-1}}\left[\E_{\pi_{k|k-1}}\left[r_k(\bv{X}_k)\right]\right] = \E_{\pi_{k:k}}\left[r_k(\bv{X}_k)\right]$. We let $\sS:={1:S}$ be the set of sources that the agent can use and $\left\{\policy{k|k-1}{(i)}\right\}_{1:N}$ be the sequence of pdfs/pmfs (i.e. the behavior) provided by the $i$-th source. While the sources do not  know the target behavior/reward of the  agent, we make the following assumption, which, as discussed in \cite{9244209}, is not restrictive in practice:
\begin{Assumption}\label{asn:source_policies}
$\DKL(\policy{k|k-1}{(i)}||p_{k|k-1})<+\infty$, $\forall k$, $\forall i\in\sS$.
\end{Assumption}
\begin{Remark} situations where the $\policy{k|k-1}{(i)}$'s are available to the agent naturally arise in a number of applications. For example, within the multi-fidelity RL framework, the sources might be simulators available to the agent and each $\policy{k|k-1}{(i)}$ might be given by the output of the $i$-th simulator.
\end{Remark}
\begin{Remark} in what follows, we use $p_{1:n|0}$ and $\pi_{1:n|0}$ for $p(\bv{x}_1,\ldots,\bv{x}_n|\bv{x}_0)$  and $\pi(\bv{x}_1,\ldots,\bv{x}_n|\bv{x}_0)$ to make equations more compact.
\end{Remark}

\section{The control problem and a link with the CCP}\label{sec:control_problem}
Let $\alpha_k^{(i)}$, $i\in\sS$, be a weight and $\boldsymbol{\alpha}_k$ be the stack of the $\alpha_k^{(i)}$'s. Then, the problem of designing agents able to dynamically choose between different sources in order to track their own target behavior while optimizing the agent-specific reward can be recast via the following
\begin{Problem}\label{prob:problem_3}
find the sequence ${\left\{\boldsymbol{\alpha}_k^\ast\right\}_{1:N}}$ solving
\begin{equation}\label{eqn:problem_3}
    \begin{aligned}
    \underset{\left\{ \boldsymbol{\alpha}_k\right\}_{1:N}}{\text{min}}
    &\DKL\left(\pi_{1:N|0}||p_{1:N|0}\right) - \sum_{k=1}^N\E_{\pi_{k-1:k-1}}\left[\tilde{r}_k(\bv{X}_{k-1})\right]\\
    s.t. & \ \pi_{k|k-1} = \sum_{i\in\sS}\alpha_k^{(i)}\policy{k|k-1}{(i)}, \ \ \ \forall k \\
    & \sum_{i\in\sS}\alpha_k^{(i)} = 1, \ \ \alpha_k^{(i)} \in \{0, 1\},   \ \ \forall k.
    \end{aligned}
\end{equation}
\end{Problem}
In Problem \ref{prob:problem_3}  the constraints formalize the fact that, at each $k$, the agent picks the behavior from one source. The first term in the cost  quantifies the discrepancy between the agent behavior and the target: hence, minimizing this term amounts to tracking the target behavior. Instead, minimizing the second term in the cost implies maximizing the agent's expected reward.
\begin{Remark} Problem \ref{prob:problem_3} captures situations where an agent needs to pick between (rather than combine) multiple sources. These situations arise in a number of applications, such as multi-fidelity RL where the agents needs to make its decision based on the output of different simulators.
\end{Remark}
\begin{Definition} $\left\{{\pi}^\ast_{k|k-1}\right\}_{1:N}$ is an optimal solution of Problem \ref{prob:problem_3} if  ${\pi}^\ast_{k|k-1}:=\sum_{i\in\sS}{\alpha_k^{(i)}}^\ast\pi_{k|k-1}^{(i)}$, with the weights being an optimal solution of (\ref{eqn:problem_3}).
\end{Definition}
We now state the CCP from \cite{9244209} in terms of Problem \ref{prob:problem_3}.
\begin{Problem}\label{prob:CCP}
the CCP is the relaxation of Problem \ref{prob:problem_3} obtained by replacing, $\forall k$, $\alpha_k^{(i)} \in \{0, 1\}$ with $\alpha_k^{(i)} \ge 0$.
\end{Problem}
In what follows, we also make use of the following definition
\begin{Definition}
  $\left\{\tilde{\pi}_{k|k-1}\right\}_{1:N}$ is an approximate solution of the CCP if $\tilde{\pi}_{k|k-1}:=\sum_{i\in\sS}{\tilde\alpha_k^{(i)}}\pi_{k|k-1}^{(i)}$, with the weights being the solution of a problem having the same constraints as the CCP and a  cost upper-bounding the cost of the CCP.
\end{Definition}
Finding the optimal solution of the CCP would involve computing, at each $k$,  $\DKL(\sum_{i\in\sS}\alpha_k^{(i)}\pi^{(i)}_{k|k-1}||p_{k|k-1})$. Even for Gaussians, this computation is analytically intractable and computationally expensive \cite{4218101,Nie_Sun_16}. Hence, in  \cite{9244209} an approximate solution for the CCP is obtained. Namely:
\begin{Lemma}\label{thm:probl_2_sol}
consider the CCP. Then {$\left\{\tilde{\pi}_{k|k-1}\right\}_{1:N}$, with} $\tilde{\pi}_{k|k-1} =\sum_{i\in\sS}{\tilde\alpha_k^{(i)}}\policy{k|k-1}{(i)}$ and 
\begin{equation}\label{eqn:sub_opt_probl_2}
    \begin{aligned}
    \tilde{\boldsymbol{\alpha}}_k \in \underset{\boldsymbol{\alpha}_k}{\text{arg min}}
    &\ \ \bv{a}_k^T(\bv{x}_{k-1})\boldsymbol{\alpha}_k\\
    s.t. & \ \sum_{i\in\sS}\alpha_k^{(i)} = 1, \ \alpha_k^{(i)}\ge 0,  
    \end{aligned}
\end{equation}
is an {approximate} solution of the CCP. In (\ref{eqn:sub_opt_probl_2}), $
 \bv{a}_k(\bv{x}_{k-1}) := [a_k^{(1)}(\bv{x}_{k-1}),\ldots,a_k^{(S)}(\bv{x}_{k-1})]^T$, while \\$ {a}_k^{(i)}(\bv{x}_{k-1})  := \DKL\left(\policy{k|k-1}{(i)}||p_{k|k-1}\right)  - \E_{\pi^{(i)}_{k|k-1}}\left[\bar{r}_k({\bv{X}_{k}})\right]$, and $\bar{r}_k({\bv{X}_{k}})$ is obtained via backward recursion as
\begin{equation}\label{eqn:policy_backward_probl_2}
\begin{split}
& \bar{r}_k(\bv{x}_k) := r_k(\bv{x}_k) + \hat{r}_k(\bv{x}_k), \\ 
& \hat{r}_k(\bv{x}_k) = - \bv{a}^T_{k+1}(\bv{x}_k)\tilde{\boldsymbol{\alpha}}_{k+1}, \ \ \hat{r}_N(\bv{x}_N) = 0.
\end{split}
\end{equation}
\end{Lemma}
\begin{Remark} the solution of (\ref{eqn:sub_opt_probl_2})  is determined, $\forall k$, based on $\bv{a}_k(\bv{x}_{k-1})$ and hence $\tilde{\boldsymbol{\alpha}}_k$ depends on $\bv{x}_{k-1}$. This dependency is omitted in what follows for notational convenience.\end{Remark}
\section{Finding a solution to Problem \ref{prob:problem_3}}
We now show that the non-optimal solution to the CCP from Lemma \ref{thm:probl_2_sol} is actually the optimal solution for Prob. \ref{prob:problem_3}.
\begin{Theorem}\label{thm:problem_3}
the approximate solution of the CCP given in Lemma \ref{thm:probl_2_sol} is also the optimal solution of Problem \ref{prob:problem_3}.
\end{Theorem}
\begin{proof} we first ({\bf Step $1$}) characterize the optimal solution of Problem \ref{prob:problem_3}.  Then ({\bf Step $2$}) we  show that the approximate solution of the CCP given by Lemma \ref{thm:probl_2_sol} must necessarily coincide with the optimal solution of Problem \ref{prob:problem_3}.

\noindent{\bf Step $1$.} In what follows we make use of the shorthand notation $\mathcal{J}_n(\pi_{1:n-1|0},p_{1:n-1|0})$ to denote $\DKL\left(\pi_{1:n-1|0}||p_{1:n-1|0}\right)- \sum_{k=1}^{n-1}\E_{\pi_{k-1:k-1}}\left[\tilde{r}_k(\bv{X}_{k-1})\right]$. The chain rule for pdfs/pmfs \cite{9244209}, the linearity of the expectation and the fact that $\DKL\left(\pi_{N|N-1}||p_{N|N-1}\right)$ depends only on $\bv{x}_{N-1}$ imply that the cost in (\ref{eqn:problem_3}) can be re-written as
\begin{align*}
&\mathcal{J}_N(\pi_{1:N-1|0},p_{1:N-1|0}) + \E_{\pi_{1:N-1|0}}\left[\DKL\left(\pi_{N|N-1}||p_{N|N-1}\right)\right] - \E_{\pi_{N-1:N-1}}\left[\tilde{r}_N(\bv{X}_{N-1}) \right] \\
&=\mathcal{J}_N(\pi_{1:N-1|0},p_{1:N-1|0}) + \E_{\pi_{N-1:N-1}}\left[\DKL\left(\pi_{N|N-1}||p_{N|N-1}\right)\right] - \E_{\pi_{N-1:N-1}}\left[\tilde{r}_N(\bv{X}_{N-1}) \right] \\
&=\mathcal{J}_N(\pi_{1:N-1|0},p_{1:N-1|0}) + \E_{\pi_{N-1:N-1}}\left[\DKL\left(\pi_{N|N-1}||p_{N|N-1}\right) - \tilde{r}_N(\bv{X}_{N-1}) \right].
\end{align*}
Hence,  (\ref{eqn:problem_3}) can be formulated as the sum of  two sub-problems:
\begin{equation}\label{eqn:split_1_problem_3}
    \begin{aligned}
    \underset{\left\{\boldsymbol{\alpha}_k\right\}_{1:N-1}}{\text{min}}
    &\mathcal{J}_N(\pi_{1:N-1|0},p_{1:N-1|0})\\
    s.t. & \ \pi_{k|k-1} = \sum_{i\in\sS}\alpha_k^{(i)}\policy{k|k-1}{(i)}, \  k\in1:N-1 \\
    & \sum_{i\in\sS}\alpha_k^{(i)} = 1, \ \alpha_k^{(i)} \in \{0, 1\}, \   k\in1:N-1,
    \end{aligned} 
\end{equation}
and, by letting 
\begin{equation*}
c_N(\bv{x}_{N-1}):=\DKL\left(\pi_{N|N-1}||p_{N|N-1}\right) - \E_{\pi_{N|N-1}}\left[\bar{r}_N(\bv{X}_{N})\right],\end{equation*}

where we set $\bar{r}_N(\bv{x}_N) = r_N(\bv{x}_N)$: 
\begin{equation}\label{eqn:split_2_problem_3}
    \begin{aligned}
    \underset{\boldsymbol{\alpha}_N}{\text{min}}
    & \ \E_{\pi_{N-1:N-1}}\left[c_N(\bv{X}_{N-1}) \right]\\
    s.t. & \ \pi_{N|N-1} = \sum_{i\in\sS}\alpha_N^{(i)}\policy{N|N-1}{(i)} \\
    & \sum_{i\in\sS}\alpha_N^{(i)} = 1, \ \alpha_N^{(i)} \in \{0, 1\}.
    \end{aligned} 
\end{equation}
That is, Problem \ref{prob:problem_3} can be approached by first solving (\ref{eqn:split_2_problem_3}) and then by taking into account its minimum to solve (\ref{eqn:split_1_problem_3}). Also, the optimal cost of (\ref{eqn:split_2_problem_3}) is $\E_{\pi_{N-1:N-1}}\left[c_N^\ast(\bf{X}_{N-1}))\right]$, where  $c_N^\ast(\bf{x}_{N-1})$ is the optimal cost obtained by solving:
\begin{equation}\label{eqn:step_intermediate_problem_3}
    \begin{aligned}
    \underset{\boldsymbol{\alpha}_N}{\text{min}}
    & \ c_N(\bv{x}_{N-1})\\
    s.t.& \ \pi_{N|N-1} = \sum_{i\in\sS}\alpha_N^{(i)}\policy{N|N-1}{(i)} \\
    & \sum_{i\in\sS}\alpha_N^{(i)} = 1, \ \alpha_N^{(i)} \in \{0, 1\}.
    \end{aligned} 
\end{equation}
The constraints of (\ref{eqn:step_intermediate_problem_3}) imply that its optimal solution  is a vector, say $\boldsymbol{\alpha}_N^\ast$, having all of its elements equal to $0$ except one element, say $j_N^\ast$, which is equal to $1$. Moreover, by evaluating the cost of (\ref{eqn:step_intermediate_problem_3}) at each feasible solution, we have that $j_N^\ast \in {\text{arg min}}_{j\in\sS} \ {a}_N^{(j)}(\bv{x}_{N-1})$, where we used the definition of  ${a}_k^{(j)}(\bv{x}_{k-1})$ given in Lemma  \ref{thm:probl_2_sol}. The cost of the sub-problem in (\ref{eqn:split_2_problem_3}) is then $\E_{\pi_{N-1:N-1}}\left[ \bv{a}_{N}^T(\bv{X}_{N-1}) \boldsymbol{\alpha}_{N}^\ast\right]$, where the vector $\bv{a}_{N}(\bv{x}_{N-1})$ is defined as in Lemma \ref{thm:probl_2_sol}. Now, the fact that the original problem has been reformulated as the sum of the two sub-problems (\ref{eqn:split_1_problem_3}) and (\ref{eqn:split_2_problem_3}) implies that the cost of Problem \ref{prob:problem_3} is equal to
\begin{equation*}
\mathcal{J}_N(\pi_{1:N-1|0},p_{1:N-1|0}) + \E_{\policy{N-1:N-1}{}}\left[\bv{a}_N^T(\bv{X_{N-1}})\boldsymbol{\alpha}_N^\ast\right],
\end{equation*}
which, by means of the chain rule for pdfs/pmfs and the definition of expectation, can be shown to be equal to
\begin{align*}
&\mathcal{J}_N(\pi_{1:N-1|0},p_{1:N-1|0})  + \E_{\pi_{N-2:N-2}}\left[\E_{\pi_{N-1|N-2}}\left[\bv{a}_N^T(\bv{X_{N-1}})\boldsymbol{\alpha}_N^\ast\right]\right]\\
&=\mathcal{J}_{N-1}(\pi_{1:N-2|0},p_{1:N-2|0})+ \E_{\pi_{1:N-2|0}}\left[\DKL(\pi_{N-1|N-2}||p_{N-1|N-2})\right]\\
&- \E_{\pi_{N-2:N-2}}\left[\tilde{r}_{N-1}(\bv{X}_{N-2})\right] + \E_{\pi_{N-2:N-2}}\left[\E_{\pi_{N-1|N-2}}\left[\bv{a}_N^T(\bv{X_{N-1}})\boldsymbol{\alpha}_N^\ast\right]\right]\\
&=\mathcal{J}_{N-1}(\pi_{1:N-2|0},p_{1:N-2|0})+ \E_{\pi_{N-2:N-2}}\left[\DKL(\pi_{N-1|N-2}||p_{N-1|N-2})\right]\\
&-\E_{\pi_{N-2:N-2}}\left[\E_{\pi_{N-1|N-2}}\left[r_{N-1}(\bv{X}_{N-1})-\bv{a}_N^T(\bv{X_{N-1}})\boldsymbol{\alpha}_N^\ast\right]\right].
\end{align*}
As a result, the problem can be split as the sum of:
\begin{subequations}
\begin{equation}\label{eqn:split_1_problem3_probl3}
    \begin{aligned}
    \underset{\left\{ \boldsymbol{\alpha}_k\right\}_{1:N-2}}{\text{min}}
    &\mathcal{J}_{N-1}(\pi_{1:N-2|0},p_{1:N-2|0})\\
    s.t. & \ \pi_{k|k-1} = \sum_{i\in\sS}\alpha_k^{(i)}\policy{k|k-1}{(i)}, \   k\in 1:N-2\\
    & \sum_{i\in\sS}\alpha_k^{(i)} = 1,  \alpha_k^{(i)}\in\{ 0,1\},   k\in1:N-2,\\
    \end{aligned} 
\end{equation}
\text{and} 
\begin{equation}\label{eqn:split_2_problem_4_probl3}
    \begin{aligned}
    \underset{\boldsymbol{\alpha}_{N-1}}{\text{min}}
    & \ \E_{\pi_{N-2:N-2}}\left[c_{N-1}(\bv{X}_{N-2}) \right]\\
    s.t. & \  \pi_{N-1|N-2} = \sum_{i\in\sS}\alpha_{N-1}^{(i)}\policy{N-1|N-2}{(i)}\\
    & \sum_{i\in\sS}\alpha_{N-1}^{(i)} = 1, \ \ \alpha_{N-1}^{(i)}\in \{0, 1\},  
    \end{aligned} 
\end{equation}
\end{subequations}
with 
\begin{align*}
&c_{N-1}(\bv{x}_{N-2}):=\DKL\left(\pi_{N-1|N-2}||p_{N-1|N-2}\right) -\E_{\pi_{N-1|N-2}}\left[ \bar{r}_{N-1}(\bv{X}_{N-1})\right],\\ &\bar{r}_{N-1}(\bv{x}_{N-1}) = {r}_{N-1}(\bv{x}_{N-1}) + \hat{r}_{N-1}(\bv{x}_{N-1}) \\
&\hat{r}_{N-1}(\bv{x}_{N-1} ):=-\bv{a}_N^T(\bv{x}_{N-1})\boldsymbol{\alpha}_N^\ast.
\end{align*}
Again, the optimal solution of  (\ref{eqn:split_2_problem_4_probl3}) is the vector $\boldsymbol{\alpha}_{N-1}^\ast$ having all of its elements equal to $0$ except element $j_{N-1}^\ast$, which is equal to $1$. Moreover, this time we have $j_{N-1}^\ast \in  {\text{arg min}}_{j\in\sS} \ {a}_{N-1}^{(j)}(\bv{x}_{N-2})$,
    and the cost of  (\ref{eqn:split_2_problem_4_probl3}) is given by $\E_{\pi_{N-2:N-2}}\left[ \bv{a}^T_{N-1}(\bv{X}_{N-2}) \boldsymbol{\alpha}_{N-1}^\ast\right]$. By iterating the above arguments, we find that, $\forall k \in  1:N-2$, Problem \ref{prob:problem_3} can always be split as the sum of two sub-problems, where the problem for the last $k$ can be solved by finding:
\begin{equation*}
    \begin{aligned}
    \underset{\boldsymbol{\alpha}_k}{\text{min}}
    & \ \DKL\left(\pi_{k|k-1}||p_{k|k-1}\right) - \E_{\pi_{k|k-1}}\left[\bar{r}_k(\bv{X}_{k})\right]\\
    s.t. & \ \pi_{k|k-1} = \sum_{i\in\sS}\alpha_k^{(i)}\policy{k|k-1}{(i)}\\
    & \sum_{i\in\sS}\alpha_k^{(i)} = 1, \ \ \alpha_k^{(i)} \in \{0, 1\}, 
    \end{aligned}
\end{equation*}
with $\bar r_k(\bv{x}_{k}) := r_k(\bv{x}_k) + \hat{r}_k(\bv{x}_k)$ and $\hat{r}_k(\bv{x}_k):= -  \bv{a}^T_{k+1}(\bv{x}_{k}) \boldsymbol{\alpha}_{k+1}^\ast$. Hence, the optimal solution of Problem \ref{prob:problem_3} is, $\forall k$, the vector  $\boldsymbol{\alpha}_{k}^\ast$ that has all of its elements equal to $0$, except one, say $j_k^\ast$, equal to $1$. Moreover, at time step $k$:
\begin{equation}\label{eqn:optimal_index}
    \begin{aligned}
  j_k^\ast \in  \underset{j\in\sS}{\text{arg min}}
    & \ {a}_k^{(j)}(\bv{x}_{k-1}).
    \end{aligned} 
\end{equation}

\noindent{\bf Step $2$.} We now consider the approximate solution from  Lemma \ref{thm:probl_2_sol}, obtained by solving, at each $k$, the problem in (\ref{eqn:sub_opt_probl_2}). Note that this is a linear problem with the standard simplex as feasibility domain. Also, note that  the vertices of the standard simplex all have entries equal to $0$, except one equal to $1$. Hence, the solution of such a problem at each $k$ is the vector, $ \tilde{\boldsymbol{\alpha}}_k$ having all of its elements equal to $0$, except element $j_k$ which is equal to $1$. Moreover, the element $j_k$ solving the problem in (\ref{eqn:sub_opt_probl_2}) corresponds to the index of the smallest element in the vector $ \bv{a}_k(\bv{x}_{k-1})$. This is, however, given by (\ref{eqn:optimal_index}) and hence, $\forall k$, $j_k = j_k^\ast$.
\end{proof}
\section{Regret analysis}
We define the regret of an agent that follows $\left\{\tilde{\pi}_{k|k-1}\right\}_{1:N}$:
\begin{Definition}\label{def:regret}
consider the cost functional $\mathcal{C}\left\{\pi_{1:N|0}\right\}:= \DKL\left(\pi_{1:N|0}||p_{1:N|0}\right) - \sum_{k=1}^N\E_{\pi_{k-1:k-1}}\left[\tilde{r}_k(\bv{X}_{k-1})\right]$. Then, the regret of the agent is
$\bar{R}_{1:N|0}:=\mathcal{C}\left\{\tilde\pi_{1:N|0}\right\} - \mathcal{C}\left\{\bar{\pi}_{1:N|0}\right\}$.
\end{Definition}
In the above definition, $\bar{\pi}_{1:N|0}$ is the solution to the following 
\begin{Problem}\label{prob:problem_1}
find the sequence $\{\bar{\pi}_{k|k-1}\}_{1:N}$ solving
\begin{equation}\label{eqn:problem_1}
    \begin{aligned}
    \underset{\left\{ \pi_{k|k-1}\right\}_{1:N}}{\text{min}}
    &\DKL\left(\pi_{1:N|0}||p_{1:N|0}\right) - \sum_{k=1}^N\E_{\pi_{k-1:k-1}}\left[\tilde{r}_k(\bv{X}_{k-1})\right]\\
    s.t. & \ \pi_{k|k-1}\in\sD, \  \ \forall k.
    \end{aligned}
\end{equation}
\end{Problem}
Problem \ref{prob:problem_1} is an infinite-dimensional relaxation of Problem \ref{prob:problem_3} and its solution is an {oracle} for such a problem.
\begin{Definition}\label{def:oracle}
 $\bar{\pi}_{0:N} := \pi_{0:0}\prod_{k=1}^N\bar{\pi}_{k|k-1}=\pi_{0:0}\bar{\pi}_{1:N|0}$,
is an {oracle} for Problem \ref{prob:problem_3} if $\mathcal{C}\left\{\bar{\pi}_{1:N|0}\right\}\le  \mathcal{C}\left\{{\hat\pi}_{1:N|0}\right\}$, $\forall \hat\pi_{1:N|0}\in\sD$.
\end{Definition}
That is, by definition, no agent can obtain a better (i.e. lower) cost than the oracle.
The next result gives an explicit expression for the oracle.
\begin{Lemma}\label{thm:probl_1_sol} the  solution of Problem \ref{prob:problem_1} is
 $\left\{ \bar\pi_{k|k-1}\right\}_{1:N}$, with
\begin{equation}\label{eqn:policy_probl_1}
 \bar\pi_{k|k-1} = \frac{p_{k|k-1}\exp\left(\bar{\rho}_k(\bv{x}_k)\right)}{\int p_{k|k-1}\exp\left(\bar{\rho}_k(\bv{x}_k)\right)d\bv{x}_k},
\end{equation}
where $\bar{\rho}_k(\bv{x}_k) := {r}_k(\bv{x}_k) + \hat{\rho}_k(\bv{x}_k)$, and 
\begin{equation}\label{eqn:policy_backward_probl_1}
\begin{split}
& \hat{\rho}_k(\bv{x}_k) := \ln\E_{p_{k+1|k}}\left[\exp(\bar{\rho}_{k+1}(\bv{X}_{k+1}))\right], \\ 
& \bar{\rho}_{N+1}(\bv{x}_{N+1}) =0.
\end{split}
\end{equation}
\end{Lemma}
\begin{proof} 
See the Appendix.
\end{proof}
\begin{Remark}
\textcolor{black}{Lemma \ref{thm:probl_2_sol} and Prop. \ref{prop:regret_upper} (see next) are stated for pdfs. These results can also be stated for pmfs by replacing integrals with sums (statements omitted here for brevity).}
\end{Remark}
\subsubsection*{Regret bound} we now give an upper bound for $\bar{R}_{1:N|0}$. In doing so, we make use of the following set of assumptions.
\begin{Assumption}\label{asn:contributors} for each $i$,  
 $\left\{\policy{k|k-1}{(i)}\right\}_{1:N}$ is the solution of
\begin{equation}\label{eqn:problem_contributors}
    \begin{aligned}
    \underset{\left\{ \pi_{k|k-1}\right\}_{1:N}}{\text{min}}
    &\DKL\left(\pi_{1:N|0}||p_{1:N|0}^{(i)}\right) - \sum_{k=1}^N\E_{\pi_{k-1:k-1}}\left[\tilde{r}^{(i)}_k(\bv{X}_{k-1})\right]\\
    s.t. & \ \pi_{k|k-1}\in\sD, \  \ \forall k,
    \end{aligned}
\end{equation}
where $p_{1:N|0}^{(i)}:= \prod_{k=1}^Np^{(i)}_{k|k-1}$ and $\tilde{r}^{(i)}_k(\bv{x}_{k-1}) = \E_{\pi_{k|k-1}}\left[r_k^{(i)}(\bv{X}_{k})\right]$ are target and reward of the $i$-th source.
\end{Assumption}
Assumption \ref{asn:contributors} states that the behavior of the sources is optimal for a task, which is (in general) different from that of the agent (the cost in (\ref{eqn:problem_contributors}) is different from the cost in Prob. \ref{prob:problem_3}). Any behavior can be formally obtained as a solution of a problem of the form of (\ref{eqn:problem_contributors}). Indeed, let  $\pi_{1:N|0}^{(i)}= \prod_{k=1}^N\policy{k|k-1}{(i)}$. Then, this behavior is the solution of a problem of the form of (\ref{eqn:problem_contributors}) having $p_{1:N|0}^{(i)}=\pi_{1:N|0}^{(i)}$ and zero reward.
Assumption \ref{asn:contributors} and Lemma \ref{thm:probl_1_sol} imply that the behavior provided by the $i$-th source is of the form:
\begin{equation}\label{eqn:behavior_contributor_optimal}
\begin{split}
\pi_{1:N|0}^{(i)} & =  \prod_{k=1}^N\frac{p_{k|k-1}^{(i)}\exp\left(\bar{r}_k^{(i)}(\bv{x}_k)\right)}{\int p_{k|k-1}^{(i)}\exp\left(\bar{r}_k^{(i)}(\bv{x}_k)\right)d\bv{x}_k},
\end{split}
\end{equation}
with $\bar{r}_{N+1}^{(i)}(\bv{x}_{N+1})  = 0$ and: 
\begin{equation}\label{eqn:contributor_cumulative_rewards}
\begin{split}
& \bar{r}_k^{(i)}(\bv{x}_k) = r_k^{(i)}(\bv{x}_k) + \hat{r}_k^{(i)}(\bv{x}_k),\\
& \hat{r}_k^{(i)}(\bv{x}_k)  = \ln\E_{p_{k+1|k}^{(i)}}\left[\exp(\bar{r}_{k+1}^{(i)}(\bv{X}_{k+1}))\right].
\end{split}
\end{equation}
We let $j_k$ be the source picked, at time step $k$, by an agent that follows $\left\{\tilde{\pi}_{k|k-1}\right\}_{1:N}$. We can now state the following
\begin{Assumption}\label{asn:log-likelihood}
{$\forall k$, there exist bounded constants, $l_k$ and $L_k$, such that $\forall \bv{x}_k\in\mathcal{X}$, ${l}_k\le \ln({p^{(j_k)}_{k|k-1}}/{p_{k|k-1}})\le L_k$.}
\end{Assumption}
\begin{Assumption}\label{asn:rewards}
{$\forall k$, there exists a bounded constant, say $R_k$, such that $\forall \bv{x}_k\in\mathcal{X}$, $\abs{r_{k}^{(j_k)}(\bv{x}_k)-r_k(\bv{x}_k)}\le R_k$.}
\end{Assumption}
\begin{Remark}
Asn. \ref{asn:log-likelihood} implies that the log-likelihood ratio $\Lambda_{1:N}:=\ln{\prod_{k=1}^N p_{k|k-1}^{(j_k)}}/{\prod_{k=1}^Np_{k|k-1}}$ is bounded. Asn. \ref{asn:rewards} is on the difference between the rewards of the sources and that of the agent. This is fulfilled when e.g. rewards are bounded.
\end{Remark}
We are now ready to state the following result.
\begin{Proposition}\label{prop:regret_upper}
let Assumption \ref{asn:contributors} - \ref{asn:rewards} hold. Then, $\bar{R}_{1:N|0} \le \sum_{k=1}^N\left(L_k-l_k+ 2R_k\right)$.
\end{Proposition}
\begin{proof}the result is proved by obtaining an explicit expression for the regret, which is then upper-bounded by leveraging Assumption \ref{asn:contributors} and Lemma \ref{thm:probl_1_sol}. The desired bound is finally obtained by means of Assumption \ref{asn:rewards} and Assumption \ref{asn:log-likelihood}. Following Def. \ref{def:regret}, $\bar{R}_{1:N|0}$ is given by: 
\begin{equation}\label{eqn:regret_def}
\DKL\left(\tilde\pi_{1:N|0}||p_{1:N|0}\right) + \sum_{k=1}^N\E_{\bar\pi_{k-1:k-1}}\left[\E_{\bar\pi_{k|k-1}}\left[{r}_k(\bv{X}_{k})\right]\right] - \DKL\left(\bar\pi_{1:N|0}||p_{1:N|0}\right) -\sum_{k=1}^N \E_{\tilde\pi_{k-1:k-1}}\left[\E_{\tilde\pi_{k|k-1}}\left[{r}_k(\bv{X}_{k})\right]\right].
\end{equation}
Note that (using $d\bv{x}_{1:N}$ as a shorthand notation for $d\bv{x}_1\ldots d\bf{x}_N $): 
\begin{align*}
&\DKL\left(\bar\pi_{1:N|0}||p_{1:N|0}\right)  \\&:= \int \bar\pi_{1:N|0}\ln\frac{\bar\pi_{1:N|0}}{p_{1:N|0}}d\bv{x}_{1:N} \\
&= \int \bar\pi_{1:N|0}\ln\frac{\tilde\pi_{1:N|0}}{p_{1:N|0}}d\bv{x}_{1:N} + \int \bar\pi_{1:N|0}\ln\frac{\bar\pi_{1:N|0}}{\tilde\pi_{1:N|0}}d\bv{x}_{1:N}\\
&\ge \int \tilde\pi_{1:N|0}\ln\frac{\tilde\pi_{1:N|0}}{p_{1:N|0}}d\bv{x}_{1:N} + \int \left(\bar\pi_{1:N|0} - \tilde\pi_{1:N|0}\right)\ln\frac{\tilde\pi_{1:N|0}}{p_{1:N|0}}d\bv{x}_{1:N}  \\
&= \DKL\left(\tilde\pi_{1:N|0}||p_{1:N|0}\right) + \int \left(\bar\pi_{1:N|0} - \tilde\pi_{1:N|0}\right)\ln\frac{\tilde\pi_{1:N|0}}{p_{1:N|0}}d\bv{x}_{1:N},
\end{align*}
where we used the fact that 
\begin{equation*}
\int \bar\pi_{1:N|0}\ln\frac{\bar\pi_{1:N|0}}{\tilde\pi_{1:N|0}}d\bv{x}_{1:N} = \DKL\left( \bar\pi_{1:N|0}||\tilde\pi_{1:N|0}\right)\ge 0.
\end{equation*}As a result, the expression in (\ref{eqn:regret_def}) is upper bounded by
\begin{equation*}
-\langle \bar\pi_{1:N|0} - \tilde\pi_{1:N|0}, \ln\frac{\tilde\pi_{1:N|0}}{p_{1:N|0}}\rangle + \sum_{k=1}^N(\E_{\bar\pi_{k-1:k-1}}\left[\tilde{r}_k(\bv{X}_{k-1})\right] - \E_{\tilde\pi_{k-1:k-1}}\left[\tilde{r}_k(\bv{X}_{k-1})\right]).
\end{equation*}
Assumption \ref{asn:contributors} and Lemma \ref{thm:probl_1_sol} imply that the behaviors  from the sources are of the form (\ref{eqn:behavior_contributor_optimal}) - (\ref{eqn:contributor_cumulative_rewards}). Hence we get that $-\langle \bar\pi_{1:N|0} - \tilde\pi_{1:N|0}, \ln\frac{\tilde\pi_{1:N|0}}{p_{1:N|0}} \rangle$ equals 
\begin{equation*}
\int \left(\tilde\pi_{1:N|0} - \bar\pi_{1:N|0}\right)\left(\ln\frac{\prod_{k=1}^N p_{k|k-1}^{(j_k)}}{\prod_{k=1}^Np_{k|k-1}}+\sum_{k=1}^N\bar{r}_k^{(j_k)}(\bv{x}_k)\right)d\bv{x}_{1:N}- \int\ \left(\tilde\pi_{1:N|0} - \bar\pi_{1:N|0}\right)\sum_{k=1}^N\hat{r}_{k-1}^{(j_k)}(\bv{x}_{k-1})d\bv{x}_{1:N},
\end{equation*}

thus implying that (\ref{eqn:regret_def}) is upper bounded by
\begin{equation}\label{eqn:regret_temporary}
\begin{split}
&\langle \tilde\pi_{1:N|0} - \bar\pi_{1:N|0}, \Lambda_{1:N}\rangle   + \E_{\tilde\pi_{1:N|0}}\left[\sum_{k=1}^N\bar{r}_k^{(j_k)}(\bv{X}_k)\right]  + \sum_{k=1}^N\E_{\bar\pi_{k:k}}\left[{r}_k(\bv{X}_{k})\right]  - \E_{\bar\pi_{1:N|0}}\left[\sum_{k=1}^N\bar{r}_k^{(j_k)}(\bv{X}_k)\right] \\
& - \sum_{k=1}^N\E_{\tilde\pi_{k:k}}\left[{r}_k(\bv{X}_{k})\right]  + \E_{\bar\pi_{1:N|0}}\left[\sum_{k=1}^N\hat{r}_{k-1}^{(j_k)}(\bv{X}_{k-1})\right] - \E_{\tilde{\pi}_{1:N|0}}\left[\sum_{k=1}^N\hat{r}_{k-1}^{(j_k)}(\bv{X}_{k-1})\right].
\end{split}
\end{equation}
Moreover, linearity of the expectation together with the fact that $\bar{r}_k^{(j_k)}(\cdot)$ only depends on the state at time step $k$ implies that: 
\begin{alignat*}{2}
&\E_{\tilde\pi_{1:N|0}}\left[\sum_{k=1}^N\bar{r}_k^{(j_k)}(\bv{X}_k)\right]   & &= \sum_{k=1}^N \E_{\tilde\pi_{1:N|0}}\left[\bar{r}_k^{(j_k)}(\bv{X}_k)\right]\\ 
& & &=\sum_{k=1}^N\E_{\tilde\pi_{k:k}}\left[\bar{r}_k^{(j_k)}(\bv{X}_k)\right]. \\
\end{alignat*}
Likewise, we have:
\begin{alignat*}{2}
&  \E_{\bar\pi_{1:N|0}}\left[\sum_{k=1}^N\bar{r}_k^{(j_k)}(\bv{X}_k)\right] & &= \sum_{k=1}^N \E_{\bar\pi_{k:k}}\left[\bar{r}_k^{(j_k)}(\bv{X}_k)\right]; 
\\& \E_{\bar\pi_{1:N|0}}\left[\sum_{k=1}^N\hat{r}_{k-1}^{(j_k)}(\bv{X}_{k-1})\right]  & &=\sum_{k=1}^N\E_{\bar\pi_{k-1:k-1}}\left[\hat{r}_{k-1}^{(j_k)}(\bv{X}_{k-1})\right]; 
\\& \E_{\tilde{\pi}_{1:N|0}}\left[\sum_{k=1}^N\hat{r}_{k-1}^{(j_k)}(\bv{X}_{k-1})\right] & & = \sum_{k=1}^N \E_{\tilde{\pi}_{k-1:k-1}}\left[\hat{r}_{k-1}^{(j_k)}(\bv{X}_{k-1})\right].
\end{alignat*}
These relationships, together with (\ref{eqn:regret_temporary}) and the first relationship in (\ref{eqn:contributor_cumulative_rewards}), lead to the following upper bound for $\bar{R}_{1:N|0}$:
\begin{equation}\label{eqn:regret_upper_temporary}
\begin{split}
&\langle \tilde\pi_{1:N|0} - \bar\pi_{1:N|0}, \Lambda_{1:N}\rangle + \sum_{k=1}^N\E_{\tilde\pi_{k:k}}\left[r_k^{(j_k)}(\bv{X}_k)- {r}_k(\bv{X}_{k})\right] +   \sum_{k=1}^N \E_{\bar\pi_{k:k}}\left[{r}_k(\bv{X}_{k}) - r_k^{(j_k)}(\bv{X}_k)\right]\\
& + \sum_{k=1}^N\left(\E_{\tilde{\pi}_{k:k}}\left[\hat{r}_k^{(j_k)}(\bv{X}_k)\right] - \E_{{\bar\pi}_{k:k}}\left[\hat{r}_k^{(j_k)}(\bv{X}_k)\right]\right)+   \sum_{k=1}^N \left(\E_{{\bar\pi}_{k-1:k-1}}\left[\hat{r}_{k-1}^{(j_k)}(\bv{X}_{k-1})\right]- \E_{\tilde{\pi}_{k-1:k-1}}\left[\hat{r}_{k-1}^{(j_k)}(\bv{X}_{k-1})\right]\right).
\end{split}
\end{equation}
Now, the sums in the last two lines of (\ref{eqn:regret_upper_temporary}) can be recast as 
$\E_{\bar\pi_{0:0}}\left[\hat{r}_0^{(j_0)}(\bv{X}_0)\right] - \E_{\tilde\pi_{0:0}}\left[\hat{r}_0^{(j_0)}(\bv{X}_0)\right]$.
Since (see Def. \ref{def:oracle}) the agent and the oracle have the same initial pdf/pmf, then this means that (\ref{eqn:regret_upper_temporary}) equals 
\begin{equation*}
\langle \tilde\pi_{1:N|0} - \bar\pi_{1:N|0}, \Lambda_{1:N}\rangle + \sum_{k=1}^N\E_{\tilde\pi_{k:k}}\left[r_k^{(j_k)}(\bv{X}_k)- {r}_k(\bv{X}_{k})\right] + \sum_{k=1}^N \E_{\bar\pi_{k:k}}\left[{r}_k(\bv{X}_{k}) - r_k^{(j_k)}(\bv{X}_k)\right].
\end{equation*}
Prop. \ref{prop:regret_upper}  follows by noticing that, for such expression: (i)  from Assumption \ref{asn:rewards}, the sums are upper bounded by  $2\sum_{k=1}^NR_k$; (ii) from Assumption \ref{asn:log-likelihood}, $\langle \tilde\pi_{1:N|0} - \bar\pi_{1:N|0}, \Lambda_{1:N}\rangle \le \sum_{k=1}^N\left(L_k-l_k\right)$.
\end{proof}
\begin{figure}[b!]
\centering
    \includegraphics[trim=0cm 3cm 0cm 0cm,clip,width=0.6\columnwidth]{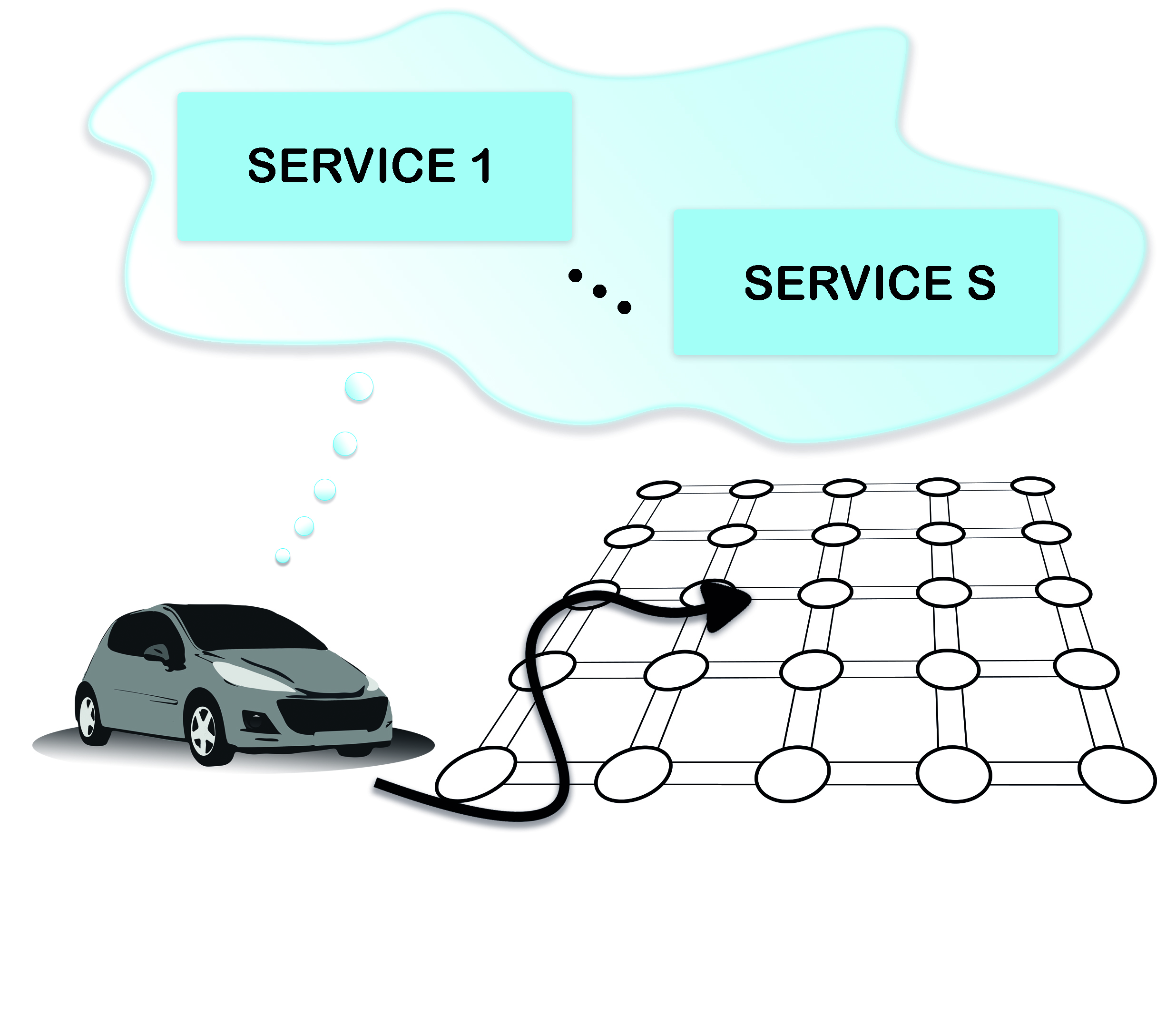}
    \caption{a connected car has access to multiple navigation services/sources to navigate within a grid-like area.}
    \label{fig:scenario}
\end{figure}
\section{Numerical Example}\label{sec:numerical_example}
We  consider the scenario of Fig. \ref{fig:scenario}, where a connected car (i.e. the agent) travels within a geographic area and has access to multiple navigation services. The services come from different providers and each of them returns a different navigation strategy. \textcolor{black}{We now make use of our results to design a mechanism enabling the agent to dynamically choose which service to follow in order to fulfill a  navigation task.} The road network we consider is the one of Fig. \ref{fig:scenario}. We set $N=8$ and the agent wishes to reach node $24$ from node $0$. We let $x_k\in\mathcal{X}:=0:24$ be the position of the agent on the graph at time step $k$, while the target behavior is given by the pmfs $p_{k|k-1}=p_k(x_k|x_{k-1})$, which might e.g. describe the preferred route of the passengers (see Fig. \ref{fig:set-up}). These pmfs can be obtained from past trips via e.g. the algorithm of \cite{8317888,8357977}. The agent has access to $3$ navigation services/sources and each provides a navigation strategy through the pmfs $\pi^{(i)}_{k|k-1}$'s, $i=1:3$ (see Fig. \ref{fig:set-up}). These pmfs are the optimal solution of a problem of the form of Problem \ref{prob:problem_1} having the reward set to $0$. Hence, following Lemma \ref{thm:probl_1_sol}, we have for each $k$ that $\pi^{(i)}_{k|k-1}=p^{(i)}_{k|k-1}$ (the pmfs \textcolor{black}{$p^{(i)}_{k|k-1}$'s} are the target behavior/route for the $i$-th service)\footnote{All pmfs, together with full size versions of the figures are given at \url{https://tinyurl.com/24ns3joy}. Code available upon request.}. 
\begin{figure}[thbp]
\centering
    \includegraphics[width=0.7\columnwidth]{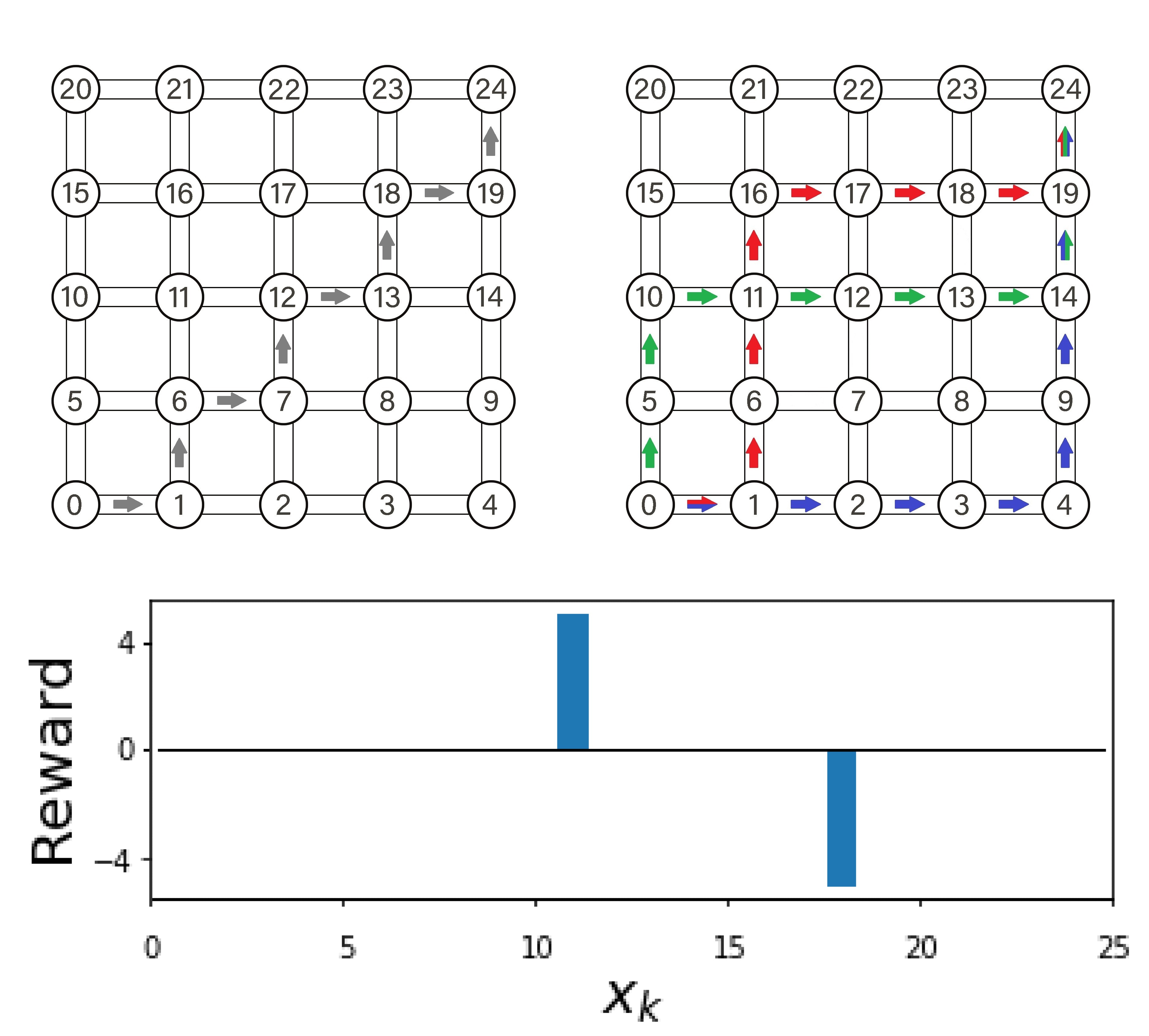}
    \caption{routes sampled from the $p_{k|k-1}$'s (top-left) and  $\pi^{(i)}_{k|k-1}$'s (top-right). Routes from the sources are highlighted with different colors (online). Bottom panel: reward used in our second set of numerical experiments - node $18$ is penalized, while $11$ is favored.}
    \label{fig:set-up}
\end{figure}

\noindent We implemented a Python script \textcolor{black}{which, given the   destination of the agent, returns: (i) the solution to Problem \ref{prob:problem_3}, i.e. $\left\{\pi_{k|k-1}^\ast\right\}_{1:N}$, by implementing the backward recursion of Lemma \ref{thm:probl_2_sol}; (ii) the route obtained by sampling, at each $k$, from $\pi_{k|k-1}^\ast$}; (iii) the cost, $\mathcal{C}\left\{\pi_{1:N|0}^\ast\right\}$ (the expression for $\mathcal{C}\left\{\cdot\right\}$ is in Def. \ref{def:regret}). We now consider two cases. First, the agent reward is  set to $0$. This reward captures situations where the agent only wishes to track the target behavior (e.g. the preferred route). This yields the pmfs (and hence the routes)  of Fig. \ref{fig:case_1} - top panels. The corresponding cost, i.e. $\mathcal{C}\left\{\pi_{1:N|0}^\ast\right\}$, is equal to $0$ and, from Prop. \ref{prop:regret_upper}, we get $\bar{R}_{1:N|0}=0$. In our second set of simulations, we  use the reward of Fig. \ref{fig:set-up}. As shown in the bottom panels of Fig. \ref{fig:case_1}, since node $11$ is now favored over node $18$ (due to e.g. risks/traffic associated to the latter node), the agent deviates from its target and the resulting route avoids node $18$. The corresponding cost is $\mathcal{C}\left\{\pi^\ast_{1:N|0}\right\} = {1.57}$ while, by means of Prop. \ref{prop:regret_upper}, we have that $\bar{R}_{1:N|0} \le {27.32}$, with this non-zero bound, due to the discrepancy between the target/rewards of the agent and these of the sources.
\begin{figure}[thbp]
\centering
    \includegraphics[width=0.85\columnwidth]{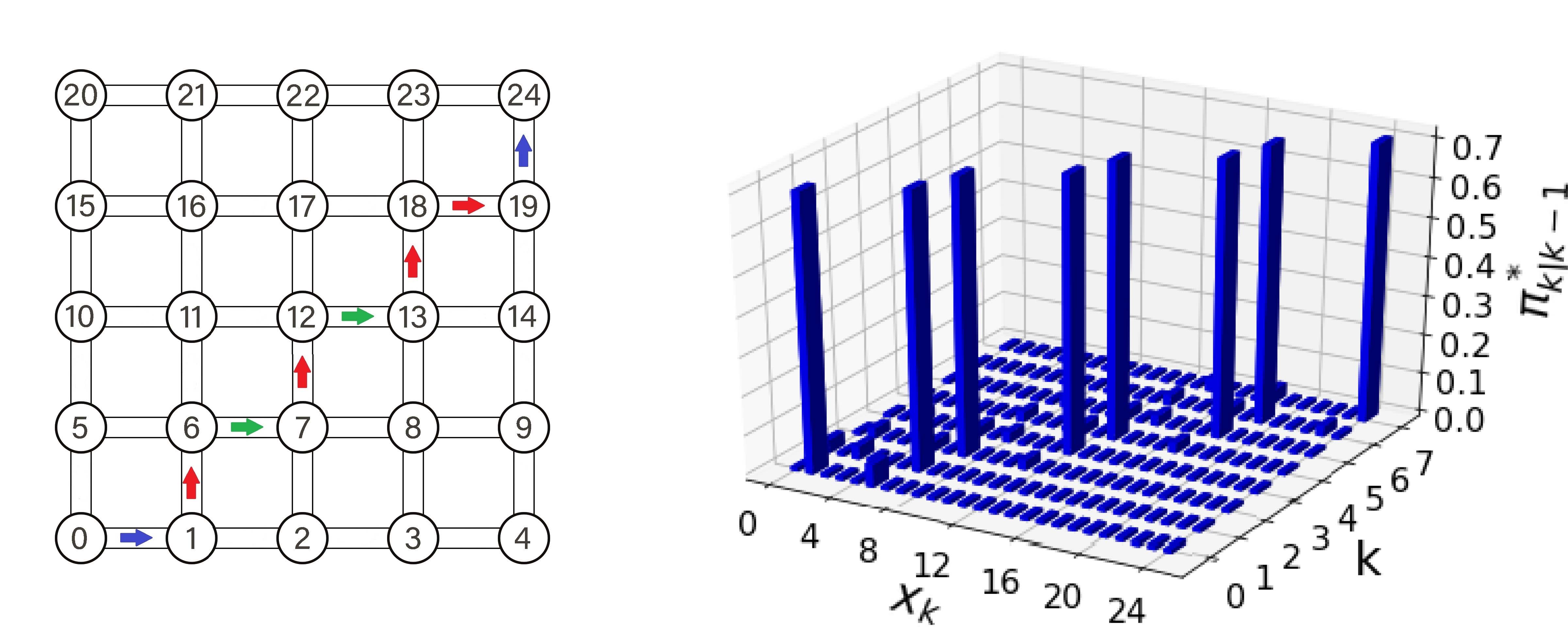}\\
    \includegraphics[width=0.85\columnwidth]{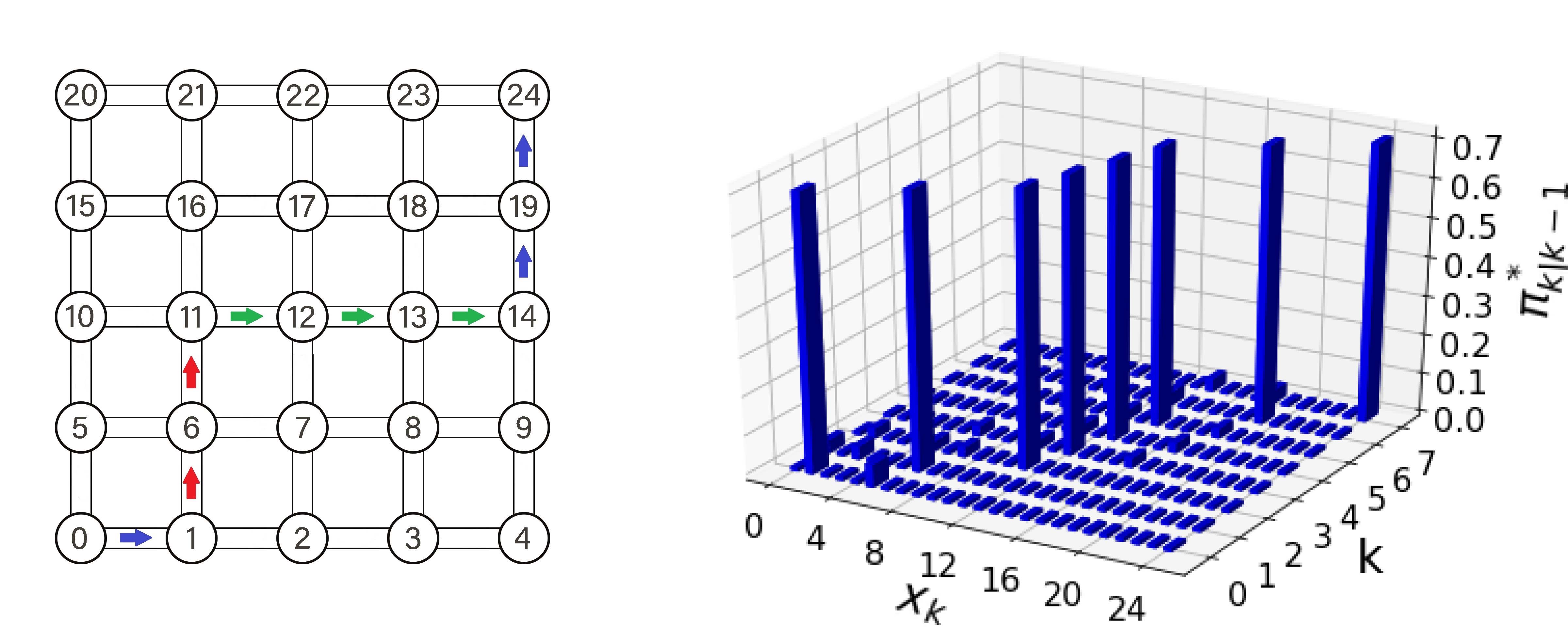}
    \caption{the agent route (left panels) and the pmfs from which it is sampled (right panels). In the top panels the reward is set to $0$, while in the bottom panels the reward is the one of Fig. \ref{fig:set-up}. The colors (online) for the routes highlight which source is selected at each $k$ by the agent (the same color code as Fig. \ref{fig:set-up} is used).}
    \label{fig:case_1}
\end{figure}
\section{Conclusions}
We considered the problem of designing agents that are able to dynamically choose between different data sources in order to tackle tracking tasks. We formulated the problem as a data-driven control problem and this led to study an integer optimal control problem. After finding the optimal solution, which determines how the use of the different sources should be orchestrated by the agent, we formalized a notion of regret. Then, by solving a relaxation of the  control problem, we gave a regret upper bound. The results were complemented via simulations on a connected cars scenario.

\appendix
\subsubsection*{Proof of Lemma \ref{thm:probl_1_sol}} we first show that Problem \ref{prob:problem_1} can be decomposed into a sequence of sub-problems, which can be solved via backward recursion. Then, the expression for the optimal solution is obtained by explicitly solving the sub-problems. We use  $\mathcal{J}_n\left\{\pi_{1:n-1|0},p_{1:n-1|0}\right\}$ to denote $\DKL\left(\pi_{1:n-1|0}||p_{1:n-1|0}\right)- \sum_{k=1}^{n-1}\E_{\pi_{k-1:k-1}}\left[\tilde{r}_k(\bv{X}_{k-1})\right]$. Following the same reasoning used to prove Theorem \ref{thm:problem_3}, we find that (\ref{eqn:problem_1}) can be  formulated as the sum of:
\begin{equation}\label{eqn:split_1_problem_1}
    \begin{aligned}
    \underset{\left\{ \pi_{k|k-1}\right\}_{1:N-1}}{\text{min}}
    & \mathcal{J}_N\left\{\pi_{1:N-1|0},p_{1:N-1|0}\right\} \\
    s.t. & \ \pi_{k|k-1}\in\sD, \  \ k \in 1 : N-1,
    \end{aligned} 
\end{equation}
and 
\begin{equation}\label{eqn:split_2_problem_1}
    \begin{aligned}
    \underset{\pi_{N|N-1}}{\text{min}}
    &\E_{\pi_{N-1:N-1}}\left[\DKL\left(\pi_{N|N-1}||p_{N|N-1}\right) - \tilde{r}_N(\bv{X}_{N-1}) \right]\\
    s.t. & \ \pi_{N|N-1}\in\sD.
    \end{aligned} 
\end{equation}
We can then approach Problem \ref{prob:problem_1} by solving (\ref{eqn:split_2_problem_1}) and then by  taking into account its solution to solve (\ref{eqn:split_1_problem_1}).  Moreover, the minimum of the above problem is $\E_{\pi_{N-1:N-1}}\left[\bar c_N(\bf{X}_{N-1}))\right]$, 
where $\bar c_N(\bf{x}_{N-1})$ is the optimal cost of:
\begin{equation}\label{eqn:step_2_problem_1}
    \begin{aligned}
    \underset{\pi_{N|N-1}}{\text{min}}
    &\mathcal{C}_N\left\{\pi_{N|N-1}\right\}\\
    s.t. & \ \pi_{N|N-1}\in\sD,
    \end{aligned} 
\end{equation}
$\mathcal{C}_N\left\{\pi_{N|N-1}\right\}:=\DKL\left(\pi_{N|N-1}||p_{N|N-1}\right) - \tilde{r}_N(\bv{x}_{N-1})$. We also note that  (\ref{eqn:step_2_problem_1}) is convex with a strictly convex twice differentiable (w.r.t. the decision variable) cost functional and we now find its minimizer by imposing the first order stationarity conditions on its Lagrangian
\begin{equation*}
\sL\left(\pi_{N|N-1},\lambda_N\right):= \int \pi_{N|N-1}\left(\ln\left(\frac{\pi_{N|N-1}}{p_{N|N-1}}-\bar{\rho}_N(\bv{x}_N)\right)\right)d\bv{x}_N + \lambda_N\left(\int \pi_{N|N-1}d\bv{x}_N -1\right),
\end{equation*}
where $\bar{\rho}_N(\bv{x}_N):= {r}_N(\bv{x}_N) + \hat{\rho}_N(\bv{x}_N)$, with $\hat{\rho}_N(\bv{x}_N) = 0$, and $\lambda_N$ is the Lagrange multiplier corresponding to the only constraint of the problem in (\ref{eqn:step_2_problem_1}). By imposing the stationarity condition with respect to $\pi_{N|N-1}$ we get that the optimal solution $\bar \pi_{N|N-1}$ must satisfy
\begin{equation*}
\ln\frac{\bar\pi_{N|N-1}}{p_{N|N-1}} - \bar{\rho}_N(\bv{x}_N) + 1 + \lambda_N = 0,
\end{equation*}
which then yields
\begin{equation*}
\bar\pi_{N|N-1} = p_{N|N-1}\frac{\exp(\bar{\rho}_N(\bv{x}_N))}{\exp(1+\lambda_N)}.
\end{equation*}
The Lagrange multiplier $\lambda_N$ can be found by imposing the stationarity condition of the Lagrangian with respect to $\lambda_N$. By imposing such a condition we get
\begin{equation*}
\exp(1+\lambda_N) = \int p_{N|N-1}{\exp(\bar{\rho}_N(\bv{x}_N))}d\bv{x}_N.
\end{equation*}
Hence: 
\begin{equation*}
\bar\pi_{N|N-1} = \frac{p_{N|N-1}\exp\left(\bar{\rho}_N(\bv{x}_N)\right)}{\int p_{N|N-1}\exp\left(\bar{\rho}_N(\bv{x}_N)\right)d\bv{x}_N}.
\end{equation*}
This yields (\ref{eqn:policy_probl_1}) - (\ref{eqn:policy_backward_probl_1}) at $k=N$. Moreover, the minimum of  (\ref{eqn:step_2_problem_1}) is given by $-\ln\E_{p_{N|N-1}}\left[\exp(\bar\rho_N(\bv{X}_N))\right]$. From this, the minimum for (\ref{eqn:split_2_problem_1}) is $-\E_{\pi_{N-1:N-1}}\left[\ln\E_{p_{N|N-1}}\left[\exp(\bar{\rho}_N(\bv{X}_N))\right]\right]$.
Next, we solve Problem \ref{prob:problem_1} for the remaining time instants. Having split Problem \ref{prob:problem_1} as the sum of the sub-problems (\ref{eqn:split_1_problem_1}) - (\ref{eqn:split_2_problem_1}) implies that its cost functional can be written as 
\begin{equation*}
\mathcal{J}_N\left\{\pi_{1:N-1|0},p_{1:N-1|0}\right\}-\E_{\pi_{N-1:N-1}}\left[\ln\E_{p_{N|N-1}}\left[\exp(\bar{\rho}_N(\bv{X}_N))\right]\right],
\end{equation*}
which, (see Property 1 of \cite{gagliardi2020probabilistic}), is equal to:
\begin{align*}
&\mathcal{J}_{N-1}\left\{\pi_{1:N-2|0},p_{1:N-2|0}\right\} - \E_{\pi_{N-2:N-2}}\left[\tilde{r}_{N-1}(\bv{X}_{N-2})\right] + \E_{\pi_{1:N-2|0}}\left[\DKL\left(\pi_{N-1|N-2}||p_{N-1|N-2}\right)\right]  \\
&-\E_{\pi_{N-1:N-1}}\left[\ln\E_{p_{N|N-1}}\left[\exp(\bar{\rho}_N(\bv{X}_N))\right]\right].
\end{align*}
Moreover: 
\begin{align*}
-\E_{\pi_{N-1:N-1}}\left[\ln\E_{p_{N|N-1}}\left[\exp(\bar{\rho}_N(\bv{X}_N))\right]\right] &= -\E_{\pi_{N-1:N-1}}\left[\hat{\rho}_{N-1}(\bv{X}_{N-1})\right]\\
&= -\E_{\pi_{N-2:N-1}}\left[\hat{\rho}_{N-1}(\bv{X}_{N-1})\right]\\
&=\int\int\pi_{N-1|N-2}\left(\bv{x}_{N-1}|\bv{x}_{N-2}\right)\pi_{N-2:N-2}(\bv{x}_{N-2})\hat{\rho}_{N-1}(\bv{x}_{N-1})d\bv{x}_{N-1}d\bv{x}_{N-2}\\
&=-\E_{\pi_{N-2:N-2}}\left[\E_{\pi_{N-1|N-2}}\left[\hat{\rho}_{N-1}(\bv{X}_{N-1})\right]\right],
\end{align*}
and
\begin{align*}
 \E_{\pi_{1:N-2|0}}\left[\DKL\left(\pi_{N-1|N-2}||p_{N-1|N-2}\right)\right]  &=\E_{\pi_{N-2:N-2}}\left[\DKL\left(\pi_{N-1|N-2}||p_{N-1|N-2}\right)\right].
\end{align*}

Hence:
\begin{align*}
&- \E_{\pi_{N-2:N-2}}\left[\tilde{r}_{N-1}(\bv{X}_{N-2})\right] + \E_{\pi_{1:N-2|0}}\left[\DKL\left(\pi_{N-1|N-2}||p_{N-1|N-2}\right)\right]  -\E_{\pi_{N-1:N-1}}\left[\ln\E_{p_{N|N-1}}\left[\exp(\bar{\rho}_N(\bv{X}_N))\right]\right]\\
&= \E_{\pi_{N-2:N-2}}\left[\DKL\left(\pi_{N-1|N-2}||p_{N-1|N-2}\right)- \E_{\pi_{N-1|N-2}}\left[{r}_{N-1}(\bv{X}_{N-1})+\hat{\rho}_{N-1}(\bv{X}_{N-1})\right]\right]\\
\end{align*}

Again, this means that the problem can be split as the sum of:
\begin{equation*}
    \begin{aligned}
    \underset{\left\{ \pi_{k|k-1}\right\}_{1:N-2}}{\text{min}}
    &\mathcal{J}_{N-1}\left\{\pi_{1:N-2|0},p_{1:N-2|0}\right\}\\
    s.t. & \ \pi_{k|k-1}\in\sD, \  \ k \in 1 :N-2,
    \end{aligned} 
\end{equation*}
and 
\begin{equation}\label{eqn:split_2_problem_4}
    \begin{aligned}
    \underset{\pi_{N-1|N-2}}{\text{min}}
    &\E_{\pi_{N-2:N-2}}\left[\mathcal{C}_{N-1}\left\{\pi_{N-1|N-2}\right\} \right]\\
    s.t. & \ \pi_{N-1|N-2}\in\sD,
    \end{aligned} 
\end{equation}
with 
\begin{align*}
&\mathcal{C}_{N-1}\left\{\pi_{N-1|N-2}\right\}: = \int \pi_{N-1|N-2}(\ln\frac{\pi_{N-1|N-2}}{p_{N-1|N-2}}-\bar{\rho}_{N-1}(\bv{x}_{N-1}))d\bv{x}_{N-1}\\
\text{and } &\bar{\rho}_{N-1}(\bv{x}_{N-1}) = {r}_{N-1}(\bv{x}_{N-1}) + \hat{\rho}_{N-1}(\bv{x}_{N-1}).
\end{align*}
The minimum for  (\ref{eqn:split_2_problem_4}) is $\E_{\pi_{N-2:N-2}}\left[\bar c_{N-1}(\bv{X}_{N-2}))\right]$, with $\bar c_{N-1}(\bv{x}_{N-2})$ being the cost obtained by solving:
\begin{equation*}
    \begin{aligned}
    \underset{\pi_{N-1|N-2}}{\text{min}}
    &\mathcal{C}_{N-1}\left\{\pi_{N-1|N-2}\right\}\\
    s.t. & \pi_{N|N-2}\in\sD.
    \end{aligned} 
\end{equation*}
That is, once the problem at time $N$ is solved, the optimization problem up to $N-1$ can be still  split as the sum of two sub-problems with the one at $k=N-1$ being independent on the sub-problem up to $k=N-2$. The sub-problem at $k=N-1$ has the same structure as (\ref{eqn:step_2_problem_1}).  Hence, by defining $\bar{\rho}_{N-1}(\bv{x}_{N-1}) = {r}_{N-1}(\bv{x}_{N-1}) + \hat{\rho}_{N-1}(\bv{x}_{N-1})$, we have that its solution is  
\begin{equation*}
\bar \pi_{N-1|N-2} = \frac{p_{N-1|N-2}\exp\left(\bar{\rho}_{N-1}(\bv{x}_{N-1})\right)}{\int p_{N-1|N-2}\exp\left(\bar{\rho}_{N-1}(\bv{x}_{N-1})\right)d\bv{x}_{N-1}}.
\end{equation*}
This gives (\ref{eqn:policy_probl_1}) - (\ref{eqn:policy_backward_probl_1}) at $k=N-1$. The optimal cost for the sub-problem at $k=N-1$ is
$$
-\E_{\pi_{N-2:N-2}}\left[\ln\E_{p_{N-1|N-2}}\left[\exp(\bar{\rho}_{N-1}(\bv{X}_{N-1}))\right]\right].
$$
We can then conclude the proof by noticing the following. By iterating the above arguments we find that, at each of the remaining time steps in $1: N-2$, Problem \ref{prob:problem_1} can always be split as the sum of two sub-problems, where the last one can be solved by finding the minimum of:
 \begin{equation}\label{eqn:split_last_problem_1}
    \begin{aligned}
    \underset{\pi_{k|k-1}}{\text{min}}
    & \DKL\left(\pi_{k|k-1}||p_{k|k-1}\right) - \E_{\pi_{k|k-1}}\left[ \bar{\rho}_k(\bv{X}_{k-1})\right]\\
    s.t. & \ \pi_{k|k-1}\in\sD,
    \end{aligned} 
\end{equation}
and $\bar{\rho}_{k}(\bv{x}_{k}) = {r}_k(\bv{x}_{k}) + \hat{\rho}_k(\bv{x}_{k})$ with $\hat{\rho}_k(\bv{x}_k) := \ln\E_{p_{k+1|k}}\left[\exp(\bar{\rho}_{k+1}(\bv{X}_{k+1}))\right]$. Hence, the solution of  (\ref{eqn:split_last_problem_1}) is 

$$
 \bar \pi_{k|k-1} = \frac{p_{k|k-1}\exp\left(\bar{\rho}_k(\bv{x}_k)\right)}{\int p_{k|k-1}\exp\left(\bar{\rho}_k(\bv{x}_k)\right)d\bv{x}_k},
$$ 
and from this we can draw the desired conclusion. \QED

\bibliographystyle{IEEEtran} 
\bibliography{SC_russobib}  


\end{document}